\documentclass[a4paper,11pt]{amsart}
\usepackage{amssymb}
\usepackage{amsmath}
\usepackage{amsthm}
\usepackage{amscd}
\theoremstyle{plain}
\newtheorem{thm}{Theorem}[section]
\newtheorem{prop}[thm]{Proposition}
\newtheorem{lem}[thm]{Lemma}
\newtheorem{cor}[thm]{Corollary}
\newtheorem{defn}[thm]{Definition}

\newtheorem{rem}[thm]{Remark}

\newtheorem{conj}[thm]{Conjecture}

\def\vol{\mathop{\mathrm{Vol}}\nolimits}

\numberwithin{equation}{section}

\title{A necessary condition for Chow semistability of polarized
toric manifolds}
\author{Hajime Ono}
\address{Department of Mathematics,
Faculty of Science and Technology,
Tokyo University of Science,
2641 Yamazaki, Noda,
Chiba 278-8510, Japan}
\email{ono\_hajime@ma.noda.tus.ac.jp}
\date{}

\begin{document}

\maketitle

\begin{abstract}
Let $\Delta\subset \mathbb{R}^n$ be an $n$-dimensional Delzant polytope.
It is well-known that there exist the $n$-dimensional compact toric
manifold $X_\Delta$ and the very ample
$(\mathbb{C}^\times)^n$-equivariant line bundle $L_\Delta$ on $X_\Delta$
associated with $\Delta$.
In the present paper, we show that if $(X_\Delta,L_\Delta^i)$ is Chow
semistable then the sum of integer points in $i\Delta$ is the constant multiple
of the barycenter of $\Delta$.
Using this result we get a necessary condition for the polarized toric
manifold $(X_\Delta,L_\Delta)$ being asymptotically Chow semistable.
Moreover we can generalize the result in \cite{fos}
to the case when $X_\Delta$ is not necessarily Fano.
\end{abstract}

\section{Introduction}


Let $X$ be a compact complex variety and $L$ an ample line bundle on $X$.
We call the pair $(X,L)$ a polarized variety.
When we study the moduli space of polarized varieties
it is important to consider stability of $(X,L)$
in the sense of geometric invariant
theory, see, for example, \cite{m}, \cite{v}. In this paper, we deal with
Chow stability of polarized varieties.


The concept of Chow stability is also significant for K\"ahler geometry:
Let $(X,L)$ be an $n$-dimensional polarized manifold. The
one of the main subjects in
K\"ahler geometry is the existence problem of K\"ahler metrics with constant
scalar curvature in the first Chern class $c_1(L)$ of $L$. In \cite{d}
Donaldson proved that if a polarized manifold $(X,L)$ admits a constant
scalar curvature K\"ahler metric (cscK metric for short) in $c_1(L)$
and if the automorphism group $\text{Aut}(X,L)$ of $(X,L)$ is discrete then
$(X,L)$ is asymptotically Chow stable. This result was extended by Mabuchi
\cite{m2} when $\text{Aut}(X,L)$ is not discrete. Namely, Mabuchi proved that
if the obstruction introduced in \cite{m1} vanishes and $(X,L)$ admits a
cscK metric in $c_1(L)$ then $(M,L)$ is asymptotically Chow polystable.
The obstruction introduced in \cite{m1} is an obstruction for $(X,L)$
to be asymptotically Chow semistable. This obstruction was reformulated by
Futaki in \cite{f} to the vanishing of a collection of integral
invariants $\mathcal F_{\text{Td}^{(1)}},\dots ,
\mathcal F_{\text{Td}^{(n)}}$.
Futaki, Sano and the author \cite{fos} also reformulated
the Mabuchi's obstruction as the vanishing of the derivation of the
Hilbert series when $(X,L)$ is a toric Fano manifold with the
anticanonical polarization.
We can compute the
Futaki's integral invariants and the derivation of the Hilbert series
for some toric Fano manifolds.
Especially, Sano, Yotsutani and the author \cite{osy}
proved the following.

\begin{thm}[\cite{osy}]\label{thm1.1}
There exists a $7$-dimensional toric K\"ahler-Einstein manifold $X$
such that $(X,-K_X)$ is asymptotically Chow unstable.
\end{thm}
Therefore, different from Donaldson's result in \cite{d},
the existence of cscK metric in $c_1(L)$ does not imply
asymptotic Chow semistability of a polarized manifold $(X,L)$
when the automorphism group $\text{Aut}(X,L)$ is not discrete.


In the present paper, we give an obstruction for Chow semistability
of polarized toric manifolds from a different viewpoint.
Let $\Delta\subset \mathbb{R}^n$ be an $n$-dimensional integral
Delzant polytope.
Namely, $\Delta$ satisfies the following conditions
(in \cite{o} a polytope satisfying these conditions are called
absolutely simple)
:
\begin{enumerate}
	\item The vertices ${\bf w}_1,\dots,{\bf w}_d$ of $\Delta$ are contained
	in $\mathbb{Z}^n$.
	\item For each vertex ${\bf w}_l$, there are $n$ edges $e_{l,1},
	\dots,e_{l,n}$ of $\Delta$ emanating from ${\bf w}_l$.
	\item The primitive vectors with respect to the edges $e_{l,1},
	\dots,e_{l,n}$ generate the lattice $\mathbb{Z}^n$ over $\mathbb{Z}$.
\end{enumerate}

It is well-known that $n$-dimensional integral Delzant polytopes
correspond to $n$-dimensional compact toric manifolds with 
$(\mathbb{C}^\times)^n$-equivariant very ample line bundles.
The reader is referred to \cite{o} for example.
Hence we have the $n$-dimensional polarized toric manifold $(X_\Delta,
L_\Delta)$ associated with $\Delta$.
The main result in this paper is the following.

\begin{thm}\label{thm1.2}
Let $\Delta\subset \mathbb{R}^n$ be an $n$-dimensional integral Delzant
polytope. If $(X_\Delta,L_\Delta^i)$ is Chow semistable for a positive
integer $i$ then
we have
\begin{equation}\label{e1.1}
\sum_{{\bf a}\in i\Delta\cap \mathbb{Z}^n}{\bf a}
=\frac{i\{\#(i\Delta\cap \mathbb{Z}^n)\}}{\vol(\Delta)}
\int_\Delta {\bf x}\, dv,
\end{equation}
where $dv$ is the Euclidean
volume form on $\mathbb{R}^n$. 
\end{thm}
\begin{rem}
It is easy to see that the following two conditions are equivalent.
\begin{itemize}
	\item The equality \eqref{e1.1} holds for some integral Delzant
	polytope $\Delta$.
	\item The equality \eqref{e1.1} holds for any
	$\mathbb{Z}^n$-translation of $\Delta$.
\end{itemize}
\end{rem}

Note here that for any $n$-dimensional integral polytope $P\subset
\mathbb{R}^n$, the number and the sum
of the integer points in $iP$
are well-behaved: There exists the polynomial $E_P(t)$ of degree $n$,
so called Ehrhart polynomial of $P$, such that
\begin{equation}\label{e1.2}
E_P(t)=\vol (P)t^n+\sum_{j=0}^{n-1}E_{P,j}t^j,\ \ E_P(i)=\#(iP\cap \mathbb{Z}
^n).
\end{equation}
Similarly it is known that there exists the $\mathbb{R}^n$-valued polynomial
${\bf s}_P(t)$ such that
\begin{equation}\label{e1.3}
{\bf s}_P(t)=t^{n+1}\int_P{\bf x}dv+\sum_{j=1}^nt^j{\bf s}_{P,j},\ \ 
{\bf s}_P(i)=\sum_{{\bf a}\in iP\cap \mathbb{Z}^n}{\bf a},
\end{equation}
see \cite{n}. Therefore by Theorem \ref{thm1.2} we easily see the following.

\begin{thm}\label{thm1.3}
Let $\Delta\subset \mathbb{R}^n$ be an $n$-dimensional integral Delzant
polytope. If $(X_\Delta,L_\Delta)$ is asymptotically Chow semistable
then \eqref{e1.1} holds for any positive integer $i$.
If \eqref{e1.1} does not hold for a positive integer $i_0$,
then there exists a positive integer $i_1$ such that
$(X_\Delta,L_\Delta^i)$ is Chow unstable for any $i\ge i_1.$
\end{thm}

We can rewrite the equality \eqref{e1.1} as
\begin{equation}\label{e1.4}
\vol(\Delta){\bf s}_\Delta(i)-iE_\Delta(i)\int_\Delta
{\bf x}dv=\sum_{j=1}^ni^j\left\{\vol(\Delta){\bf s}_{\Delta,j}-E_{\Delta,j-1}
\int_\Delta {\bf x}dv\right\}={\bf 0}.
\end{equation}

Hence we have the following obstructions for asymptotic Chow semistability of
$(X_\Delta,L_\Delta)$.

\begin{cor}
Let $\Delta\subset \mathbb{R}^n$ be an $n$-dimensional integral Delzant
polytope. If $(X_\Delta,L_\Delta)$ is asymptotically Chow semistable
then 
\begin{equation}\label{e1.5}
\mathcal F_{\Delta,j}:=\vol(\Delta){\bf s}_{\Delta,j}-E_{\Delta,j-1}\int_\Delta
{\bf x}dv\in \mathbb{R}^n
\end{equation}
vanishes for each $j=1,\dots,n$.
\end{cor}

On the one hand, these vectors $\mathcal F_{\Delta,j}$ are regarded as
characters on the Lie algebra of the $n$-dimensional torus. On the other hand,
Futaki's integral invariants $\mathcal F_{\text{Td}^{(p)}},\ p=1,\dots,n$
are characters on the Lie algebra of holomorphic vector fields on $X$,
\cite{f}.

\begin{conj}\label{conj1.5}
\begin{equation}\label{e1.6}
\textup{Lin}_{\mathbb{C}}\{\mathcal F_{\Delta,j},j=1,\dots,n\}=
\textup{Lin}_{\mathbb{C}}\{\mathcal F_{\textup{Td}^{(p)}}|
_{\mathbb{C}^n},p=1,\dots,n\}
\subset \mathbb{C}^n,
\end{equation}
where $\textup{Lin}_{\mathbb{C}}$ stands for the linear hull in
$\mathbb{C}^n$.
\end{conj}

We next consider the special case.
An $n$-dimensional integral polytope $P\subset \mathbb{R}^n$ is called
reflexive if $P$ satisfies the following conditions:

\begin{enumerate}
	\item For each codimension $1$ face $F\subset P$, there is an
	${\bf n}_F\in \mathbb{Z}^n$ with $F=\{{\bf x}\in P\,|\,
	\langle {\bf x},{\bf n}_F\rangle=1\}$,
    \item The origin ${\bf 0}\in \mathbb{R}^n$ is contained in the interior
    of $P$.
\end{enumerate}

It is well-known that reflexive
Delzant polytopes correspond to toric Fano manifolds with
the anticanonical polarization. 

\begin{cor}\label{cor1.6}
Let $\Delta\subset \mathbb{R}^n$ be an $n$-dimensional reflexive Delzant
polytope. Then the following conditions are equivalent.

(1) $(X_\Delta,L_\Delta)$ is asymptotically Chow semistable.

(2) For all positive integer $i$, the equality
\begin{equation}\label{e1.7}
{\bf s}_\Delta(i)=\frac{iE_\Delta(i)}{\vol (\Delta)}\int_\Delta
{\bf x}dv={\bf 0}
\end{equation}
holds.
\end{cor}

We next observe the relation between asymptotic Chow semistability
of $(X_\Delta,L_\Delta)$ and the derivative of the Hilbert series.
Let $\Delta$ be an $n$-dimensional integral Delzant polytope
and ${\bf w}_1,\dots,{\bf w}_d\in \mathbb{Z}^n$ the vertices of $\Delta$.
We put
$$\mathcal C(\Delta):=\{r_1({\bf w}_1,1)+\cdots+r_d({\bf w}_d,1)
\in \mathbb{R}^{n+1}\,|\, r_1,\dots,r_d\ge 0\}$$
and
\begin{equation}\label{e1.8}
C_\Delta(x_1,\dots,x_{n+1}) :=
\sum _{(a_1,\dots,a_{n+1})\in \mathcal C(\Delta)\cap \mathbb{Z}^{n+1}}
x_1^{a_1}\cdots x_{n+1}^{a_{n+1}}
=:
\sum_{{\bf a}\in \mathcal C(\Delta)\cap \mathbb{Z}^{n+1}}
{\bf x}^{\bf a}.
\end{equation}
We call $C_\Delta$ the Hilbert series of $\Delta$. Since
\begin{equation}\label{e1.9}
\begin{pmatrix}
\frac{\partial C_\Delta}{\partial x_1}(1,\dots,1,t)\\
\vdots\\
\frac{\partial C_\Delta}{\partial x_n}(1,\dots,1,t)\\
\end{pmatrix}
=\sum_{i=1}^\infty {\bf s}_\Delta(i)t^i
\end{equation}
holds,
the derivative of the Hilbert series at $(1,\dots,1,t)$ can be regarded as
the generating function of ${\bf s}_\Delta(i)$.
By \eqref{e1.9} and Theorem  \ref{thm1.2}, we see the following.

\begin{cor}\label{cor1.7}

If $(X_\Delta,L_\Delta)$ is asymptotically Chow semistable then we have
\begin{equation}\label{e1.10}
\begin{pmatrix}
\frac{\partial C_\Delta}{\partial x_1}(1,\dots,1,t)\\
\vdots\\
\frac{\partial C_\Delta}{\partial x_n}(1,\dots,1,t)\\
\end{pmatrix}
=\left(\sum_{i=1}^{\infty} i E_\Delta (i)t^i\right)
\frac{\int_\Delta {\bf x}dv}{\vol (\Delta)}.
\end{equation}
Moreover when $\Delta$ is reflexive
\begin{equation}\label{e1.11}
\begin{pmatrix}
\frac{\partial C_\Delta}{\partial x_1}(1,\dots,1,t)\\
\vdots\\
\frac{\partial C_\Delta}{\partial x_n}(1,\dots,1,t)\\
\end{pmatrix}
={\bf 0}
\end{equation}
holds.
\end{cor}
The equality \eqref{e1.11} is equivalent to
the necessary condition for asymptotic Chow semistability
of toric Fano manifolds proved in \cite{fos}.
Hence Corollary \ref{cor1.7} is a generalization of the result in \cite{fos}
to the case when $X_\Delta$ is not necessarily Fano.

This paper is organized as follows. In Section $2$, we first
review the definition
and some results
of semistability in the sense of geometric invariant theory \cite{git}.
We next give the definition of Chow form of projective varieties.
In Section $3$, we give a proof of our main theorem, Theorem \ref{thm1.2},
based on the results in \cite{gkz}. We also provide some examples of
Chow unstable polarized toric manifolds.

{\bf Acknowledgements} : The author is supported by MEXT, Grant-in-Aid for
Young Scientists (B), No. 20740032.


\section{Preliminaries}


Let $G$ be a reductive Lie group. Suppose that $G$ acts a
complex vector space $V$ linearly. We call a nonzero vector $v\in V$
$G$-semistable
if the closure of the orbit $Gv$ does not contain the origin.
Similarly we call $p\in \mathbb{P}(V)$ $G$-semistable if any representative
of $p$ in $V\setminus\{{\bf 0}\}$ is $G$-semistable.
It is well-known that there is the following
good criterion for $v$ being $G$-semistable, see
\cite{git}.

\begin{prop}[Hilbert-Mumford criterion, \cite{git}]\label{prop2.1}
$p\in \mathbb{P}(V)$ is $G$-semistable if and only if
$p$ is $H$-semistable for each maximal torus $H\subset G$.
\end{prop}


Hence it is important to study $G$-semistability
when $G$ is isomorphic to an algebraic torus $(\mathbb{C}^\times )^n$.
Let $G$ be isomorphic to $(\mathbb{C}^\times )^n$.
Then a $G$-module $V$ is decomposed as
\begin{equation}\label{e2.1}
V=\sum_{\chi\in \chi(G)}V_\chi,\ \ V_\chi:=\{v\in V\,|\,
t\cdot v=\chi(t)v,\ \forall t\in G\},
\end{equation}
where $\chi(G)\simeq \mathbb{Z}^n$ is the character group of the torus $G$.

\begin{defn}\label{def2.2}
Let $v=\sum_{\chi\in \chi(G)}v_\chi$ be a nonzero vector in $V$.
The weight polytope $\textup{Wt}\,_G(v)\subset \chi(G)\otimes_{\mathbb{Z}}
\mathbb{R}$ of $v$ is the convex hull of $\{\chi\in \chi(G)\,|\,
v_\chi\not={\bf 0}\}$ in $\chi(G)\otimes_{\mathbb{Z}}
\mathbb{R}$.
\end{defn}

The following fact about $G$-semistability is standard.

\begin{prop}\label{prop2.3}
Let $G$ be isomorphic to $(\mathbb{C}^\times)^n$. Suppose that $G$ acts
a complex vector space $V$ linearly. Then a nonzero vector $v\in V$ is
$G$-semistable if and only if the weight polytope $\textup{Wt}\,_G(v)$
contains the origin.
\end{prop}

Let $G=(\mathbb{C}^\times)^{n+1}$ and $H$ be the subtorus
\begin{equation}\label{e2.2}
H=\{(t_1,\dots,t_n,(t_1\cdots t_n)^{-1})\,|\,
(t_1,\dots,t_n)\in (\mathbb{C}^\times)^n\}\simeq (\mathbb{C}^\times)^n.
\end{equation}
Then the weight polytope $\text{Wt}\,_H(v)\subset \chi(H)\otimes_{\mathbb{Z}}
\mathbb{R}\simeq \mathbb{R}^n$ equals to $\pi(\text{Wt}\,_G(v))$, where
the linear map $\pi:\mathbb{R}^{n+1}\to \mathbb{R}^n$ is given as
$(x_1,\dots,x_n,x_{n+1})\mapsto (x_1-x_{n+1},\dots,x_n-x_{n+1})$.

Therefore we see the following.

\begin{prop}\label{prop2.4}
If $v$ is $H$-semistable then there exists $t\in \mathbb{R}$
such that $(t,\dots,t)\in \textup{Wt}\,_G(v)$.
\end{prop}


We next define the Chow form of irreducible projective varieties.
See \cite{gkz} for more detail.

\begin{defn}\label{def2.5}
Let $X\subset \mathbb{C}P^N$ be an $n$-dimensional irreducible subvariety of
degree $d$. It is easy to see that the subset $Z_X$ of the
Grassmannian $\textup{Gr}(N-n-1,\mathbb{C}P^N)$ defined by
$$Z_X=\{L\in \textup{Gr}(N-n-1,\mathbb{C}P^N)\,|\,
L\cap X\not=\emptyset\}$$
is an irreducible hypersurface of degree $d$.
Hence $Z_X$ is given by the vanishing of a degree $d$ element $R_X\in
\mathbb{P}(\mathcal B_d(N-n-1,\mathbb{C}P^N))$, where
$\mathcal B(N-n-1,\mathbb{C}P^N)=\oplus_{d}\mathcal B_d(N-n-1,\mathbb{C}P^N)$
is the graded coordinate ring of the Grassmannian. We call $R_X$ the
Chow form of $X$.
\end{defn}

Since the special linear group $SL(N+1,\mathbb{C})$ acts naturally on
$\mathcal B_d(N-n-1,\mathbb{C}P^N)$, we can consider the 
$SL(N+1,\mathbb{C})$-stability
of the Chow form $R_X$.


\begin{defn}\label{def2.6}
Let $X\subset \mathbb{C}P^N$ be an $n$-dimensional irreducible subvariety
of degree $d$. We call $X$ Chow semistable if the Chow form $R_X$
is $SL(N+1,\mathbb{C})$-semistable.
When $X$ is not Chow semistable $X$ is called Chow unstable.
\end{defn}

\begin{defn}\label{def2.7}
Let $(X,L)$ be a polarized manifold. $(X,L)$ is called 
asymptotically Chow semistable when $\Psi_i(X)\subset \mathbb{P}(H^0(X;L^i)^*)$
is Chow semistable for each $i\gg 1$. Here $\Psi_i:X\to 
\mathbb{P}(H^0(X;L^i)^*)$
is the Kodaira embedding.
\end{defn}

\section{Asymptotic Chow semistability of polarized toric manifolds}


In this section,
we first introduce a necessary condition for Chow semistability of
$n$-dimensional irreducible projective subvariety
$X\subset \mathbb{C}P^N$. This condition is very simple and easy.
However this leads us to our main theorem when $X$ is toric.

Let $(\mathbb{C}^\times)^{N+1}\subset GL(N+1,\mathbb{C})$ be the
$(N+1)$-dimensional torus consists of invertible diagonal matrices.
Then the subtorus $H$ defined by \eqref{e2.2} is a maximal torus of
$SL(N+1,\mathbb{C})$. 
On the one hand, Proposition \ref{prop2.4} implies the following.

\begin{prop}\label{prop3.2}
Let $X\subset \mathbb{C}P^N$ be an irreducible subvariety.
If $X$ is Chow semistable, then there exists $t\in \mathbb{R}$ such that
\begin{equation}\label{e3.2}
(t,\dots,t)\in \textup{Wt}\,_{(\mathbb{C}^\times)^{N+1}}(R_X)\subset
\textup{Aff}_{\mathbb{R}}(\textup{Wt}\,_{(\mathbb{C}^\times)^{N+1}}(R_X)).
\end{equation}
\end{prop}

On the other hand, Gelfand, Kapranov and Zelivinsky showed the
following.

\begin{prop}[\cite{gkz}, Chapter $6$, Proposition $3.8$]\label{prop3.1}
Let $X\subset \mathbb{C}P^N$ be an $n$-dimensional irreducible subvariety.
Suppose that $X$ does not contained in any projective hyperplane.
Then we have
\begin{equation}\label{e3.1}
\dim \textup{Aff}_{\mathbb{R}}(\text{Wt}\,_{(\mathbb{C}^\times)^{N+1}}(R_X))=
N+1-\dim\{t\in (\mathbb{C}^\times )^{N+1}\,|\, tX=X\},
\end{equation}
where $\textup{Aff}_{\mathbb{R}}(\textup{Wt}
\,_{(\mathbb{C}^\times)^{N+1}}(R_X))$
is the affine hull of the weight polytope
$\textup{Wt}\,_{(\mathbb{C}^\times)^{N+1}}(R_X)$ of $R_X$
in $\mathbb{R}^{N+1}$.
\end{prop}

Therefore, as $\dim\{t\in (\mathbb{C}^\times )^{N+1}\,|\, tX=X\}$ is
larger, it is harder that $X$ is Chow semistable.


We next investigate Chow semistability of projective varieties
defined as follows. Let $A:=\{{\bf a}_1,\dots,{\bf a}_{N+1}\}$
be a finite subset in
$\mathbb{Z}^n$. Suppose that $A$ affinely generates the lattice $\mathbb{Z}^n
\subset \mathbb{R}^n$ over $\mathbb{Z}$. Then the closure of
$$X_A^0:=
\{[{\bf x}^{{\bf a}_1}:\dots:{\bf x}^{{\bf a}_{N+1}}]\,|\,
{\bf x}\in (\mathbb{C}^\times)^n\}
\subset
\mathbb{C}P^N$$
is an $n$-dimensional subvariety of $\mathbb{C}P^N$. We denote by $X_A=
\overline{X_A^0}$.

\begin{prop}[\cite{gkz}, Chapter $7$, Proposition $1.11$ $\&$ \cite{ksz}]
\label{prop3.3}
\begin{equation}\label{e3.3}
\begin{split}
&\textup{Aff}_{\mathbb{R}}(\text{Wt}\,_{(\mathbb{C}^\times)^{N+1}}(R_{X_A}))\\
&=\left\{(\varphi_1,\dots,\varphi_{N+1})\in \mathbb{R}^{N+1}\,|\,
\sum_{j=1}^{N+1}\varphi_j=(n+1)!\vol(Q),\ 
\sum_{j=1}^{N+1}\varphi_j{\bf a}_j=(n+1)!\int_Q{\bf x}dv
\right\}.
\end{split}
\end{equation}
\end{prop}

Hence, by Propositions \ref{prop3.2} and \ref{prop3.3}, we get a necessary
condition for Chow semistability of $X_A$.

\begin{thm}\label{thm3.4}
If $X_A$ is Chow semistable then we have
\begin{equation}\label{e3.4}
\sum_{j=1}^{N+1}{\bf a}_j=\frac{N+1}{\vol (Q)}\int_Q{\bf x}dv.
\end{equation}
\end{thm}
\begin{proof}
If $X_A$ is Chow semistable then
there exists $t\in \mathbb{R}$ satisfying
\begin{equation}\label{e3.5}
(N+1)t=(n+1)!\vol (Q),\ \ t\sum_{j=1}^{N+1}{\bf a}_j=(n+1)!\int_Q{\bf x}dv
\end{equation}
by Propositions \ref{prop3.2} and \ref{prop3.3}. Hence \eqref{e3.4} holds.
\end{proof}

\begin{proof}[Proof of Theorem \ref{thm1.2}]
Let $\Delta\subset \mathbb{R}^n$ be an $n$-dimensional integral Delzant
polytope and $A:=i\Delta\cap \mathbb{Z}^n$. Then $X_A$ is the image
of the Kodaira embedding $X_\Delta\to \mathbb{P}(H^0(X_\Delta,L_\Delta^i)^*)$.
Therefore if $(X_\Delta,L_\Delta^i)$ is Chow semistable
then we have
\begin{align*}
{\bf s}_\Delta(i) &= 
\frac{E_\Delta(i)}{\vol (i\Delta)}\int_{i\Delta}{\bf x}dv
=\frac{iE_\Delta(i)}{\vol (\Delta)}\int_{\Delta}{\bf x}dv
\end{align*}
by Theorem \ref{thm3.4}.
\end{proof}

\begin{proof}[Proof of Corollary \ref{cor1.6}]
By \cite{f}, when $X_\Delta$ is asymptotically Chow semistable,
the
Futaki invariant of $X_\Delta$, in this case $\int_\Delta{\bf x}dv$, vanishes. 
Therefore ${\bf s}_\Delta(i)={\bf 0}$ for any positive integer $i$.
Conversely, suppose that ${\bf s}_\Delta(i)=\int_\Delta{\bf x}dv={\bf 0}$
for any positive integer $i$. From the result of Wang and Zhu \cite{wz}, 
there exists an K\"ahler-Einstein metrics on $X_\Delta$. Moreover
we see that the derivative of Hilbert series $C_\Delta$ vanishes
by \eqref{e1.9}.
Hence the integral invariant $\mathcal F_{\text{Td}^{(p)}}$ vanish for 
each $p=1,\dots,n$ \cite{fos}. Therefore by the result of Mabuchi \cite{m2},
$X_\Delta$ is asymptotically Chow semistable.
\end{proof}


We give some examples of polarized toric manifolds and investigate
Chow semistability.
We first investigate polarized toric surfaces.
The equalities \eqref{e1.2} and \eqref{e1.3} imply the following.

\begin{lem}\label{lem3.5}
Let $P\subset \mathbb{R}^2$ be an integral polygon.
Then we have
\begin{equation}\label{e3.6}
\begin{split}
E_P(t)=\vol(P)t^2+(E_P(2)-E_P(1)-3\vol(P))t
+2E_P(1)-E_P(2)+2\vol(P)
\end{split}
\end{equation}
and
\begin{equation}\label{e3.7}
\begin{split}
{\bf s}_P(t)=t^3\int_P {\bf x}dv+\frac{t^2}{2}\left({\bf s}_P(2)
-2{\bf s}_P(1)-6\int_P {\bf x}dv\right)+\frac{t}{2}\left(4{\bf s}_P(1)
-{\bf s}_P(2)+4\int_P{\bf x}dv\right).
\end{split}
\end{equation}
\end{lem}

For example, let $\Delta_k$ be the convex hull of
$\{(0,0),(0,1),(1,1),(k,0)\}\subset \mathbb{R}^2$ for $k\ge 2$.
The corresponding toric surface $X_{\Delta_k}$ is the $(k-1)$-th
Hirzebruch surface. $X_{\Delta_k}$ is not Fano for $k\ge 3$.
By Theorem \ref{thm1.2}
we see Chow unstability of $(X_{\Delta_k},L_{\Delta_k}^i)$.

\begin{prop}
For each integers $k\ge 2$ and $i\ge 1$,
$(X_{\Delta_k},L_{\Delta_k}^i)$ is Chow unstable.
\end{prop}
\begin{proof}
It is easy to see that
\begin{equation}\label{e3.8}
\int_{\Delta_k}{\bf x}dv=\frac16
\begin{pmatrix}
k^2+k+1\\
k+2
\end{pmatrix},
\end{equation}
\begin{equation}\label{e3.9}
\vol(\Delta_k)=\frac12 (k+1),
\end{equation}
\begin{equation}\label{e3.10}
E_{\Delta_k}(1)=k+3,\ E_{\Delta_k}(2)=3k+6
\end{equation}
and
\begin{equation}\label{e3.11}
{\bf s}_{\Delta_k}(1)=\frac12
\begin{pmatrix}
k^2+k+2\\
4
\end{pmatrix},\ 
{\bf s}_{\Delta_k}(2)=\frac12
\begin{pmatrix}
5k^2+5k+8\\
2k+16
\end{pmatrix}
\end{equation}
hold. Hence by Lemma \ref{lem3.5},
\begin{equation}\label{e3.12}
E_{\Delta_k}(t)=\frac12\{(k+1)t^2+(k+3)t+2\}
\end{equation}
and
\begin{equation}\label{e3.13}
{\bf s}_{\Delta_k}(t)=\frac{t}{12}
\begin{pmatrix}
2(k^2+k+1)t^2+3(k^2+k+2)t+k^2+k+4\\
2(k+2)t^2+12t+8-2k
\end{pmatrix}.
\end{equation}
Therefore
\begin{align*}
\vol(\Delta_k){\bf s}_{\Delta_k}(i)-iE_{\Delta_k}(i)\int_{\Delta_k}{\bf x}dv
&=\frac{i(i+1)k(k-1)}{24}
\begin{pmatrix}
k-1\\
-2
\end{pmatrix}\not=
\begin{pmatrix}
0\\
0
\end{pmatrix}
\end{align*}
holds for any $k\ge 2$ and $i\ge 1$.
\end{proof}

\begin{rem}
In this case, note that $\mathcal F_{\Delta_k,1}$ equals to
$\mathcal F_{\Delta_k,2}$:
$$\mathcal F_{\Delta_k,1}=\mathcal F_{\Delta_k,2}=
\frac{k(k-1)}{24}
\begin{pmatrix}
k-1\\
-2
\end{pmatrix}.
$$
When $k=2$, that is, $X_{\Delta_2}$ is the one point blow-up of
the projective plane,
we can easily calculate the Futaki's integral invariants
$\mathcal F_{\textup{Td}^{(1)}}|_{\mathbb{C}^2}$ 
and $\mathcal F_{\textup{Td}^{(2)}}|_{\mathbb{C}^2}$:
$$
\mathcal F_{\textup{Td}^{(1)}}|_{\mathbb{C}^2}=C_1
\begin{pmatrix}
1\\
-2
\end{pmatrix},\ \ 
\mathcal F_{\textup{Td}^{(2)}}|_{\mathbb{C}^2}=C_2
\begin{pmatrix}
1\\
-2
\end{pmatrix},\ \ C_1,C_2\not=0.
$$
Therefore, in this case \eqref{e1.6} is right:
$$\textup{Lin}_{\mathbb{C}}\{\mathcal F_{\Delta_2,1},\mathcal F_{\Delta_2,2}\}=
\textup{Lin}_{\mathbb{C}}\{\mathcal F_{\text{Td}^{(1)}|_{\mathbb{C}^2}},
\mathcal F_{\text{Td}^{(2)}}|_{\mathbb{C}^2}\}=\mathbb{C}
\begin{pmatrix}
1\\
-2
\end{pmatrix}\subset \mathbb{C}^2.
$$
\end{rem}

We next consider the example given by Nill and Paffenholz in \cite{np}.
It is the toric Fano $7$-fold corresponding to the reflexive Delzant polytope
\begin{equation}\label{e3.14}
\Delta_{NP}:=\{{\bf x}\in \mathbb{R}^7\,|\,
\langle {\bf x},{\bf v}_i\rangle \ge -1,i=1,\dots,12\}.
\end{equation}
Here ${\bf v}_1,\dots,{\bf v}_{12}\in \mathbb{R}^7$ are given by
\begin{align*}
(
{\bf v}_1  \ {\bf v}_2  \ 
{\bf v}_3  \ 
{\bf v}_4  \ 
{\bf v}_5  \ 
&{\bf v}_6  \ 
{\bf v}_7  \ 
{\bf v}_8  \ 
{\bf v}_9  \ 
{\bf v}_{10}  \ 
{\bf v}_{11}  \ 
{\bf v}_{12}
) \\
&=
\left(\begin{array}{rrrrrrrrrrrr}
1 & 0 & 0 & 0 & 0 & -1 & 0 & 0 & 0 & 0 & 0 & 0
\\
0 & 1 & 0 & 0 & -1 & 0 & 0 & 0 & 0 & 0 & 0 & 0
\\
0 & 0 & 1 & -1 & 0 & 0 & 0 & 0 & 0 & 0 & 0 & 0
\\
0 & 0 & 0 & 0 & 0 & 0 & 1 & 0 & 0 & -1 & 0 & 0
\\
0 & 0 & 0 & 0 & 0 & 0 & 0 & 1 & 0 & -1 & 0 & 0
\\
0 & 0 & 0 & 0 & 0 & 0 & 0 & 0 & 1 & -1 & 0 & 0
\\
0 & 0 & 0 & -1 & -1 & -1 & 0 & 0 & 0 & 2 & 1 & -1
\end{array}\right).
\end{align*}

Sano, Yotsutani and the author showed the following in \cite{osy}.

\begin{thm}[\cite{osy}]
Let $\Delta_{NP}$ be the $7$-dimensional reflexive Delzant polytope
given above. Then 
\begin{equation}
\int_{\Delta_{NP}}{\bf x}dv={\bf 0}
\end{equation}
and
\begin{equation}
\begin{pmatrix}
\frac{\partial C_{\Delta_{NP}}}{\partial x_1}(1,\dots,1,t)\\
\vdots\\
\frac{\partial C_{\Delta_{NP}}}{\partial x_n}(1,\dots,1,t)\\
\end{pmatrix}
\not={\bf 0}
\end{equation}
hold.
\end{thm}

Therefore $X_{\Delta_{NP}}$ admits K\"ahler-Einstein metrics
by the theorem of Wang and Zhu \cite{wz},
but $(X_{\Delta_{NP}},-K_{X_{\Delta_{NP}}})$ is asymptotically
Chow unstable. Moreover, by Theorem \ref{thm1.3}, we see that
there exists a positive integer $i_1$ such that
$(X_{\Delta_{NP}},(-K_{\Delta_{NP}})^i)$ is
Chow unstable for any $i\ge i_1$.

\end{document}